\documentclass[12pt]{amsart}

\usepackage{graphicx}

\usepackage{subcaption} 
\usepackage{tikz}
\usepackage{tikz-cd}
\tikzcdset{scale cd/.style={every label/.append style={scale=#1},
    cells={nodes={scale=#1}}}}
\usepackage{hyperref}

\hypersetup{bookmarks=true,
	unicode=true,
	colorlinks=true,
	citecolor=black,
	linkcolor=black,
	urlcolor=black,
	plainpages=false,
	pdfpagelabels=true}

 \usepackage{url}	
 \allowdisplaybreaks 

\usepackage{pgf}

\usepackage{xcolor}

\usepackage{comment} 

\usepackage{tikz}
\usetikzlibrary{arrows}
\usetikzlibrary{decorations.markings,arrows,automata,arrows,backgrounds,snakes}
\usepackage{tikz-cd} 
\usepackage{pgf}
\usetikzlibrary{babel}

\usepackage{appendix}

\usepackage{mathrsfs}
\usepackage{dsfont}

\usepackage{mathtools}

\usepackage[shortlabels]{enumitem}

\newtheorem{theorem}{Theorem}[section]
\newtheorem{proposition}[theorem]{Proposition}
\newtheorem{lemma}[theorem]{Lemma}
\newtheorem{corollary}[theorem]{Corollary}

\theoremstyle{definition}
\newtheorem{definition}[theorem]{Definition}
\newtheorem{example}[theorem]{Example}

\theoremstyle{remark}
\newtheorem{remark}[theorem]{Remark}

\usepackage{color}
\definecolor{darkgreen}{cmyk}{1,0,1,.2}
\definecolor{m}{rgb}{1,0.1,1}

\newdimen\theight
\def\TeXref#1{%
             \leavevmode\vadjust{\setbox0=\hbox{{\tt
                     \quad\quad  {\small \textrm #1}}}%
             \theight=\ht0
             \advance\theight by \lineskip
             \kern -\theight \vbox to
             \theight{\rightline{\rlap{\box0}}%
             \vss}%
             }}%

\begin{document}

 \title[Sheaf theoretic approach to Lefschetz calculus]{Sheaf theoretic approach to Lefschetz calculus}
 \thanks{This research was partially supported by TEORÍA DE MORSE, TOPOLOGÍA, ANÁLISIS Y DINÁMICA - GENERACIÓN DE CONOCIMIENTO 2020 (PID2020-114474GB-I00).
}

\author[A. Majadas-Moure \and D. Mosquera
     ]{%
	Alejandro O. Majadas-Moure \and David Mosquera-Lois
}
              
\address{
		 Alejandro O. Majadas-Moure \\
		 Departamento de Matemáticas, Universidade de Santiago de Compostela, SPAIN}
		\email{alejandro.majadas@usc.es} 
   \address{ David Mosquera-Lois \\ Departamento de Matemáticas, Universidade de Vigo, SPAIN }\email{david.mosquera.lois@uvigo.es}
   \email{david.mosquera.lois@usc.es} 
		

\begin{abstract} 
We lift the Lefschetz number from an algebraic invariant of maps between spaces to an invariant of morphisms of data over the spaces. 
\end{abstract}



\maketitle


\section{Introduction}

Both in Applied and Pure Mathematics, it is becoming more popular to study not only (algebraic) invariants of spaces, but also (algebraic) invariants of data over these spaces. A very insightful instance of this approach shows up in Euler Calculus (see, for example, \cite{C-G-M,B-G} and the references therein) where the Euler integral may be seen as a lift of the classical Euler characteristic $\chi$ from an invariant of spaces to an invariant of data over spaces.

The Lefschetz number of a continuous map between topological spaces $f\colon X\to X$ is a fundamental invariant in Mathematics. It may be seen as a generalization of the Euler-Poincaré characteristic since the Lefschetz number of the identity map coincides with the Euler-Poincaré characteristic. The importance of the Lefschetz number led to the development of  Lefschetz calculus (see \cite{M-M}), that is, integration with respect to the Lefschetz number.

In this work we aim to lift the Lefschetz number from an invariant of maps between spaces to an invariant of morphisms of data structures over the spaces. In doing so, we answer positively a question raised in \cite{C-G-M}.

In order to accomplish our goal, the first step is to formalize the idea of data structures over spaces. We do so in two novel different complementary approaches and we show the strong relation between them (Theorems \ref{teor integral}, \ref{thm:axiomatization_lefschetz}). Let $X$ be a definable or tame cell complex (for example a finite simplicial complex in $\mathbb{R}^n$) and let $f\colon X\to X$ be a homeomorphism. In the first approach, which is a continuation of our previous work (\cite{M-M}), we define data over a space as a counting function $f\colon X\to \mathbb{Z}$ and the aggregation of information corresponds to integration with respect to the combinatorial Lefschetz number. The second approach encodes data over $X$ as a special kind of (constructible) sheaf $\mathcal{F}$ over $X$ and the algebraic-topological invariant obtained corresponds to computing a sheaf-theoretic Lefschetz number of $f$.

The main accomplishments of this work are the following. 

First, we define the novel notions of combinatorial Lefschetz number (see Subsection \ref{subsec_comb_lefschetz_number}) and sheaf-theoretic Lefschetz number (see Subsection \ref{subsec:sheaf_theoretic_lefschetz}) of a homeomorphism between definable cell complexes and we establish a relation between  them (Theorem \ref{iso comp comb}).

Second, we introduce the two novel approaches to lifting the Lefschetz number from maps between spaces to morphisms of data over spaces: the combinatorial one in Subsection \ref{subsec:integration_comb_lefschetz_number} and the sheaf theoretic one in Subsection \ref{subsec:sheaf_theoretic_lefschetz}. The former one has the advantage of computability and simplicity while the latter  not only provides an enlightenment of Lefschetz calculus (which we began in \cite{M-M}) but also proves invaluable in guaranteeing well-definedness of the combinatorial approach (Corollary \ref{coro:well-definedness_combinatorial_number}) and its topological invariance (Corollary \ref{coro:invariance_comb_lefschetz}). 

Third, we prove a representation theorem (Theorem \ref{teor integral}) which establishes a strong bond between aggregation of data information (integration with respect to the combinatorial Lefschetz number) and the sheaf-theoretic number.  Theorem \ref{teor integral} may be restated in very simplified terms as follows:

 \medskip
{\noindent {\rm\bf Theorem \ref{teor integral} (Representation theorem).}
Let $X$ be a nice cell complex  and let $f\colon X \to X$ be a cellular homeomorphism. If $\mathcal{F}$ is a nice (constructible) sheaf on $X$ compatible with $f\colon X \to X$, then the sheaf-theoretic Lefschetz number can be computed by aggregating data in a combinatorial way: \begin{equation}\label{eq:corresopndence_sheaves_integrals_functions_a}
        L_c(X,f,\mathcal{F})=\int_X h\mathrm{d}\varLambda f,
    \end{equation}
where $h\colon X \to \mathbb{N}$ is the nice function associated to $\mathcal{F}$.
Conversely, if $h\colon X \to \mathbb{N}$ is a nice function,   its associated nice sheaf $\mathcal{F}$ satisfies Equation (\ref{eq:corresopndence_sheaves_integrals_functions_a}).} \medskip

Moreover, we obtain Barrow's rule for integration with respect to Lefschetz number (Theorem \ref{coro:barrow}):
    \begin{equation*}
        \int_X h \mathrm{d}\varLambda f= L_c(X,f,\oplus_{j>0}\tilde{\mathbb{R}}^j_{h^{-1}(j)})-L_c(X,f,\oplus_{j<0}\tilde{\mathbb{R}}^{-j}_{h^{-1}(j)}).
    \end{equation*}

Fourth, we study to what extent the Representation theorem (Theorem \ref{teor integral}) characterizes the sheaf-theoretic Lefschetz number. Let be $\mathcal{C}$ the family of triples $(X,f,\mathcal{F})$ where $X$ is a ``nice'' cell complex, $\mathcal{F}$ is a ``nice'' (constructible) sheaf on $X$ and $f$ a cellular homeomorphism compatible with $\mathcal{F}$. We prove Theorem \ref{thm:axiomatization_lefschetz}, which may be stated in a very simplified way as follows:

\medskip
{\noindent {\rm\bf Theorem \ref{thm:axiomatization_lefschetz}.}
The sheaf-theoretic Lefschetz number is the only map between $\mathcal{C}$ and $\mathbb{Z}$ that satisfies:
    \begin{enumerate}
        \item A cofibration axiom to compute it inductively on the skeleta.  
\item The Representation theorem (Theorem \ref{teor integral}) on wedge of spheres and graphs. That is, if $X$ is a finite wedge of $n$-spheres or a graph, $\mathcal{F}$ is a nice sheaf and $f$ is a cellular homeomorphism compatible with $\mathcal{F}$, then
\begin{equation*}
    L_c(X,f,\mathcal{F})=\int_X h d\varLambda f,
\end{equation*}
where $h$ is the constructible function associated to the sheaf.
    \end{enumerate}} \medskip

Compare Theorem \ref{thm:axiomatization_lefschetz} with  \cite{A-B}, where Lefschetz number between cell complexes is characterized.

\section{Definitions of the Lefschetz number}\label{sec:def_lefschetz_number}
In this section, we introduce the combinatorial and sheaf-theoretic Lefschetz numbers. In order to do so, we first introduce the notions of definable cell complexes and $f$-$c$-constructible sheaves.

\subsection{Definable cellular structures}

We introduce the definable structures we will work with. For a more detailed account on $o$-minimal topology we refer the reader to \cite{Dries}. Recall that an \textit{$o$-minimal} structure over $\mathbb{R}$ is a collection $\mathscr{A}=\{\mathscr{A}_n\}_{n\in \mathbb{N}}$ so that the following properties hold:
\begin{enumerate}
\item $\mathscr{A}_n$ is an algebra of subsets of $\mathbb{R}^n$ for each $n\in \mathbb{N}$.
\item The family $\mathscr{A}$ is closed with respect to cartesian products and canonical projections.
\item The subset $\{(x,y)\in\mathbb{R}^2, x<y\}$ is in $\mathscr{A}_2$.
\item The family $\mathscr{A}_1$ consists of all finite unions of points and open intervals of $\mathbb{R}$.
\item Every $\mathscr{A}_n$ contains all the algebraic subsets of $\mathbb{R}^n$.
\end{enumerate}
Given an o-minimal structure, we will say that a set $A\subset \mathbb{R}^n$ is \textit{definable} if $A\in \mathscr{A}_n$.

A {\em generalized simplicial complex} is a finite collection $K $ of open simplices in $\mathbb{R}^n$, satisfying that given two open simplices in the complex, the intersection of their closures is the empty set or the closure of an open simplex in the complex. By {\em incomplete subcomplex} of a simplicial complex we mean a union of open simplices (which it is not necessarily closed). We illustrate a generalized simplicial complex in Figure \ref{fig:generalized_simplicial_complex}, where, as we will do along the paper, we will abuse of notation and denote by $X$ both the geometric realization and the underlying  simplicial complex.
\begin{figure}[h!]
    \centering
\includegraphics[width=0.25\linewidth]{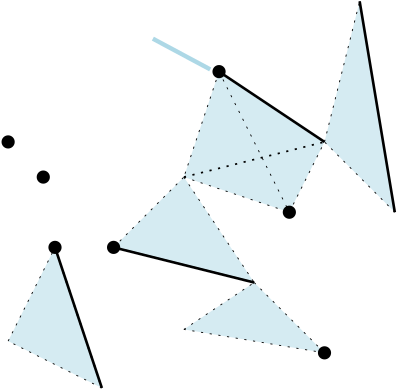}
    \caption{Generalized simplicial complex.}
\label{fig:generalized_simplicial_complex}
\end{figure}

We will use o-minimal structures that contain the semi-linear sets so that, in this way, the generalised simplicial complexes are definable. Moreover, it holds the following triangulation theorem:

\begin{theorem}[{Definable triangulation theorem \cite{Dries}}] \label{thm:triangulation} 
	Let $X\subset \mathbb{R}^n$ be a definable set and let $\{X_i\}_{i=1}^{m}$ be a finite family of definable subsets of $X$. Then there exists a definable triangulation of $X$ compatible with the collection of subsets.
\end{theorem}

Let be $X$ a finite (therefore compact) CW complex. It can be embedded in some $\mathbb{R}^n$. Therefore, it makes sense to define a {\em definable cellular complex} $X$ as a finite CW complex such that all its open cells are definable. As a consequence, definable cellular complexes have definable skeletons in all dimensions. We define an {\em 
incomplete subcomplex} of $X$ as a union of open cells of $X$.  An incomplete subcomplex of a compact generalized simplicial complex $X$ is an example of a {\em incomplete subcomplex} of $X$. The reader may think in terms of generalized simplicial complexes when reading Section \ref{sec:sheaf-theoretic_Lefschetz_calculus}.

\subsection{The combinatorial Lefschetz number}\label{subsec_comb_lefschetz_number}
We briefly recall the definition of the combinatorial Lefschetz number for homeomorphisms between simplicial complexes. We refer the reader to \cite{M-M} for a detailed treatment. Let $f:X\rightarrow X$ be a homeomorphism of a simplicial complex to itself. Let $U\subset X$ be a definable $f$-invariant set ($f(U)= U$). Let $\overline{U}$ be its closure in $X$. We define the combinatorial Lefschetz number of $f:U\subset X\rightarrow X$ (relative to $U$) in several steps:

\begin{enumerate}
    \item We consider a triangulation $(L,\overline{K},K)$ of $X$ compatible with $(X, \overline{U}, U)$ (the existence of this triangulation guaranteed by Theorem \ref{thm:triangulation}) and a map $\tilde{f}\colon L\to L$ induced by $f$.
    \item We construct a simplicial approximation $\tilde{f}^{\text{\rm simp}}\colon L\to L$ of $\tilde{f}$ in $L$ (with respect to a possibly finer simplicial structure in $L$).
    \item Let $C_*(\tilde{f}^{\text{\rm simp}}):C_*(L)\rightarrow C_*(L)$ be the induced simplicial chain map. In each dimension $p$, we will say that an \textit{incomplete basis of $C_*(U)$} consists of the $p$-simplices $\sigma^p$ such that the open simplex corresponding to $\sigma^p$ is an open simplex of $U$. Thus, we can restrict the matrix of the map $C(\tilde{f}^{\text{\rm simp}})$ to the square submatrix of coefficients $(i,j)$ such that $\sigma_i^p$ and $\sigma_j^p$ belong to the incomplete basis of $C_*(U)$. Then, $\varLambda^c(C(\tilde{f}^{\text{\rm simp}}),C_*(U))$ is defined as the alternating sum of the traces of these submatrices in the different dimensions of the complex $X$.
    \item Finally, we set $\varLambda_c (U,f)_X=\varLambda^c(C(\tilde{f}^{\text{\rm simp}}),C_*(U)).$
\end{enumerate}

We extend the definition to cellular homeomorphisms. Let $X$ be a definable finite cellular complex. Let $f:X\rightarrow X$ be a cellular homeomorphism and consider an incomplete subcomplex $U\subset X$ that is $f$-invariant and locally compact. Then, by Theorem \ref{thm:triangulation}  there exists a definable triangulation of $X$ compatible with $U$, call it $(Y,K,g)$ (it does exist even if $U$ is not locally compact). We define the Lefschetz combinatorial number of $f$ in $U$ as
    \begin{equation*}
\varLambda_c(U,f)_X=\varLambda_c(K,g)_Y.
    \end{equation*}
For simplicity of the notation, we will denote $\varLambda_c(U,f)$ instead of $\varLambda_c(U,f)_X$.
We will prove later that the combinatorial Lefschetz number of $f$ in $U$ is well-defined (Corollary \ref{coro:well-definedness_combinatorial_number}). 

The following lemma follows from the additivity of the combinatorial Lefschetz number  for generalized simplicial complexes (see \cite{M-M}) and the definable triangulation theorem (Theorem \ref{thm:triangulation}):

\begin{lemma}[Additivity of the combinatorial Lefschetz number]\label{lemma:additivity_combinatorial_lefschetz_number}
    Let $X$ be a definable cellular complex and let $U$ and $V$ be two disjoint definable and $f$-invariant subsets. Then, $\varLambda_c( U\cup V,f)_X=\varLambda_c(U,f)_X+\varLambda_c(V,f)_X$.
\end{lemma}

\subsection{$c$-constructible sheafs}
Let be $X$ a definable cellular complex. We say that a sheaf $\mathcal{F}$ of finite dimensional real vector spaces on $X$ is a \textit{$c$-constructible sheaf} if: \begin{itemize}
    \item there exists a partition $\{X_i\}$ of definable locally compact incomplete subcomplexes of $X$ such that, for all $i$, $\mathcal{F}_{X_i}$ is constant on each connected component (this is the case for example of a constructible sheaf where the connected components of each $X_i$ where the sheaf is locally constant are simply-connected).
\end{itemize} If $f:X\rightarrow X$ is a homeomorphism we say that the $c$-constructible sheaf $\mathcal{F}$ is \textit{$f$-constructible} if the $X_i$ are $f$-invariant.


To make this work more self-contained and readable by a wider audience, we recall some results concerning locally constant sheaves (see \cite{C-G-M, Dimca, Hainaut, Iversen, Tennison}). 

\begin{lemma}\label{lemas previos}
    \begin{enumerate}
        \item[(A)]\label{inv cte} If $\mathcal{F}$ is a constant sheaf on $X$ and $f:X\rightarrow X$ is a map, then $f^{\ast}(\mathcal{F})=\mathcal{F}$.
        \item[(B)]\label{loc cte en conexo}
    If $X$ is a simply-connected space and $\mathcal{F}$ is a locally constant sheaf, then $\mathcal{F}$ is the constant sheaf. 
        \item[(C)]\label{hom inducido} Let $f:X\rightarrow Y$ be a proper map and let $\mathcal{F}$ be a sheaf on $Y$. There exists a homomorphism $f^{\ast}:H_c^{\ast}(Y,\mathcal{F})\rightarrow H_c^{\ast}(X, f^{\ast}\mathcal{F})$ induced by $f$ in cohomology with compact supports.
        \item[(D)] \label{coho cte}
    Let be $A$ a definable subset of $X$ and $\tilde{\mathbb{{R}}}_{A}$ the constant sheaf on $A$. Then $H_c^{\ast}(A,\tilde{\mathbb{{R}}}_{A})\cong H_c^{\ast}(A,\mathbb{R})$. 
    \end{enumerate}
\end{lemma}

We are interested in applying Lemma \ref{hom inducido} (C) to the case of a homeomorphism $f\colon X\to X$ and a $f$-$c$-constructible sheaf $\mathcal{F}$. However, when working with locally constant sheaves which are not constant, even if the map $f$ respects the locally constant components, in general it is not true that $f^{\ast}\mathcal{F}$ equals $\mathcal{F}$. We circumvent this problem by defining an {\em associated $f$-$c$-constructible sheaf }$\mathcal{F}'$ to $\mathcal{F}$. We define
    \begin{equation*}
        \mathcal{F}'=(i_1)_!i_1^\ast\mathcal{F}\oplus (i_2)_!i_2^\ast\mathcal{F}\oplus\ldots\oplus (i_k)_!i_k^\ast\mathcal{F},
    \end{equation*}
where $\{X_1, X_2,\ldots X_k\}$ are the locally compact incomplete subcomplexes of the partition of $X$ where the sheaf is locally constant and the $i_j\colon X_j\to X$ are the inclusion maps.

Note that $\mathcal{F}'$ is automatically a $f$-$c$-constructible sheaf with the same definable partition as the original $f$-$c$-constructible sheaf.

\subsection{The associated $f$-$c$-constructible sheaf $\mathcal{F}'$}
The goal of this subsection is to prove that under certain conditions, there is a natural isomorphism between $f^\ast(\mathcal{F}')$ and $\mathcal{F}'$. 
\begin{lemma}\label{primer lema inv}
    Let $A$ be a topological space which equals the disjoint union of two open subsets $W_1$ and $W_2$. Let $\mathcal{F}$ be a sheaf on $A$. Then there exists a canonical isomorphism between $\mathcal{F}$ and $(i_1)_!i_1^\ast\mathcal{F}\oplus (i_2)_!i_2^\ast\mathcal{F}$, where $i_1$ and $i_2$ are the inclusions.
\end{lemma}
\begin{proof}
    As $W_1$ and $W_2$ are also closed, we have $(i_1)_!=(i_1)_\ast$ and $(i_2)_!=(i_2)_\ast$. On the other hand, as they are open, we have $i_1^\ast\mathcal{F}=\mathcal{F}_{|W_1}$ and $i_2^\ast\mathcal{F}=\mathcal{F}_{|W_2}$. In consequence, if we take $V$ an open set of $A$ we have
    \begin{equation*}
        \bigl((i_1)_!i_1^\ast\mathcal{F}\bigr)(V)=\bigl((i_1)_\ast\mathcal{F}_{|W_1}\bigr)(V)=\mathcal{F}_{|W_1}(V\cap W_1)=\mathcal{F}(V\cap W_1).
    \end{equation*}
    The same result is obtained for $i_2$.
    Finally,
    \begin{equation*}
        \bigl((i_1)_!i_1^\ast\mathcal{F}\oplus (i_2)_!i_2^\ast\mathcal{F}\bigr)(V)=\mathcal{F}(V\cap W_1)\oplus \mathcal{F}(V\cap W_2)\cong\mathcal{F}(V). \qedhere
    \end{equation*}
\end{proof}
As a consequence of this result we obtain the following corollary.
\begin{corollary}\label{cor inv}
    Let $X$ be a topological space, A a locally closed subspace of $X$ that is the disjoint union of two open subsets $W_1$ and $W_2$ of $A$. Let $\mathcal{F}$ be a sheaf on $X$ and $i_1:W_1\rightarrow A$, $i_2:W_2\rightarrow A$, $j:A\rightarrow X$, $j_1:W_1\rightarrow X$ $j_2:W_2\rightarrow X$ the inclusions. Then there is a natural isomorphism
    \begin{equation*}
        j_!j^\ast\mathcal{F}\cong (j_1)_!j_1^\ast\mathcal{F}\oplus(j_2)_!j_2^\ast \mathcal{F}
    \end{equation*}
\end{corollary}
\begin{proof}
By Lemma~\ref{primer lema inv}, we have
\begin{equation*}
    j^\ast\mathcal{F}\cong(i_1)_!i_1^\ast j^\ast\mathcal{F} \oplus (i_2)_!i_2^\ast j^\ast\mathcal{F}= (i_1)_!j_1^\ast\mathcal{F}\oplus (i_2)_! j_2^\ast \mathcal{F}
\end{equation*}
and then
\begin{equation*}
    j_!j^\ast\mathcal{F}\cong j_!(i_1)_!j_1^\ast\mathcal{F}\oplus j_!(i_2)_! j_2^\ast \mathcal{F}\cong (j_1)_!j_1^\ast\mathcal{F}\oplus(j_2)_!j_2^\ast \mathcal{F},
\end{equation*}
where the last isomorphism comes from the uniqueness of the extension by zero (see \cite{Tennison}[8.6, p. 63]).
\end{proof}

\begin{lemma}\label{lema tocho inv}
    Let $X$ be a topological space, let $g\colon X\to X$ be a continuous map, and let $A\subset X$ be a locally closed subspace such that $g(A)\subset A$. If $\mathcal{G}$ is a constant sheaf on $A$, then we have
    \begin{equation*}
        g^\ast i_!\mathcal{G}\cong i_!\mathcal{G}.
    \end{equation*}
\end{lemma}
\begin{proof}
    Let $p:A\rightarrow \{\ast\}$ be the constant map, $V$ the vector space given by the constant sheaf and $\mathcal{V}$ the constant sheaf on $\{\ast\}$ whose sections on the point equal $V$. We have $\mathcal{G}=p^\ast\mathcal{V}$. Now, as the diagram 
    \[\begin{tikzcd}
	A && A \\
	& {\{\ast\}}
	\arrow["{g_{|A}}", from=1-1, to=1-3]
	\arrow["p"', from=1-1, to=2-2]
	\arrow["p", from=1-3, to=2-2]
\end{tikzcd}\]
commutes, we have $g_{|A}^\ast\mathcal{G}=\mathcal{G}$. Let us consider now the commutative diagram
\[\begin{tikzcd}
	A & A \\
	X & X
	\arrow["{g_{|A}}", from=1-1, to=1-2]
	\arrow["i"', from=1-1, to=2-1]
	\arrow["i"', from=1-2, to=2-2]
	\arrow["g", from=2-1, to=2-2]
\end{tikzcd}\]
As $A$ is locally closed in $X$ we have, by \cite{Iversen}[Chapter II, 6.13]
\begin{equation*}
    g^\ast i_!\mathcal{G}\cong i_! g_{|A}^\ast\mathcal{G},
\end{equation*}
and so
    \begin{equation*}
        g^\ast i_!\mathcal{G}\cong i_!\mathcal{G}. \qedhere
    \end{equation*}
\end{proof}

\begin{theorem}\label{inv f'}
    Let $X$ be a definable cellular complex and let $f\colon X\to X$ be a cellular homeomorphism. Let  $\mathcal{F}$ be a $f$-$c$-constructible sheaf which respects the constant components of the sheaf. Then there is a natural isomorphism between $f^\ast(\mathcal{F}')$ and $\mathcal{F}'$.
\end{theorem}

\begin{proof}
    We want to prove that 
    \begin{equation*}
        f^\ast \bigl((i_1)_!i_1^\ast\mathcal{F}\oplus (i_2)_!i_2^\ast\mathcal{F}\oplus\ldots\oplus (i_k)_!i_k^\ast\mathcal{F}\bigr)\cong (i_1)_!i_1^\ast\mathcal{F}\oplus (i_2)_!i_2^\ast\mathcal{F}\oplus\ldots\oplus (i_k)_!i_k^\ast\mathcal{F}.
    \end{equation*}
    As the inverse image preserves the direct sum, it suffices to see that $f^\ast\bigl((i_j)_!i_j^\ast\mathcal{F}\bigr)$ is isomorphic to $ (i_j)_!i_j^\ast\mathcal{F}$ for all $j=1,\ldots, k$.

By Lemma~\ref{loc cte en conexo} we have that these sheaves are constant on each connected component of the subcomplexes $X_j$. Since  these components are finite, then they are open and we can apply Corollary~\ref{cor inv} to separate each sheaf $ (i_j)_!i_j^\ast\mathcal{F}$ into its constant components and, as they are $f$-invariant ($f$ respects these constant components), we can apply Lemma~\ref{lema tocho inv} to finish the proof.
\end{proof}

\subsection{The sheaf-theoretic Lefschetz number}\label{subsec:sheaf_theoretic_lefschetz}

Let be $X$ a definable cellular complex, $\mathcal{F}$ a $c$-constructible sheaf on $X$ and $f$ a cellular homeomorphism that respects the constant components of $\mathcal{F}$ and the partition that we consider to define the associated sheaf. We define the \textit{sheaf-theoretic Lefschetz number of} $f$ 
\textit{asociated to} $\mathcal{F}$ as
    \begin{equation*}
        L_c (X,f,\mathcal{F})=\sum_n (-1)^n \mathrm{tr}(f^{\ast},H^n_c(X,\mathcal{F}')).
    \end{equation*}
We will set the notation $L_c (X,f)=L_c (X,f,\tilde{\mathbb{R}})$.

Note that the way of assigning $\mathcal{F}'$ to $\mathcal{F}$ could be different if we considered a different partition. Theorem~\ref{teor integral} will show that the Lefschetz number is independent of this choice and therefore is well-defined. As a consequence of this, in the rest of the paper we will denote by $\mathcal{F}$ both $\mathcal{F}$ and $\mathcal{F}'$. 

\begin{remark}
    We will see later that the sum $\sum_n (-1)^n \mathrm{tr}(f^{\ast},H^n_c(X,\mathcal{F}'))$ converges since  $H^n_c(X,\mathcal{F}')$ vanishes for a sufficiently large $n$.
\end{remark}

\section{Well-definedness and topological invariance of the combinatorial Lefschetz number} \label{sec:sheaf-theoretic_Lefschetz_calculus}
The purpose of this section is to prove that the combinatorial Lefschetz number in the setting of cellular complexes and incomplete subcomplexes is well-defined. Moreover, we also obtain a topological invariance result as well as other results that will be of interest in the future.

We begin by checking that the sheaf-theoretic Lefschetz number satisfies a certain additivity rule. 

\begin{proposition}[Dimensional additivity of sheaf-theoretic Lefschetz number]\label{adit sop comp}
    Let $X$ be a cellular complex and $D$ a locally compact subcomplex. Consider a cellular homeomorphism $f:X\rightarrow X$ such that $D$ is $f$-invariant (this is $f(D)=D$). Then, the sheaf-theoretic Lefschetz number equals the sum of the Lefschetz numbers on each dimension:
    \begin{equation*}
    \mathrm{L}_c(D,f)=\sum_i \mathrm{L}_c(D^{(i)}-D^{(i-1)},f)
    \end{equation*}
\end{proposition}

\begin{proof}
    Suppose $D$ has dimension $n$. Then, $D^{(n)}-D^{(n-1)}$, that is, the union of all the open cell of maximal dimension of $D$, is open in $D$. Then, we have the following long exact sequence
\begin{equation*}
    \ldots\rightarrow H_c^m(D^{(n)}-D^{(n-1)})\rightarrow H_c^m (D) \rightarrow H_c^m(D^{(n-1)})\rightarrow\ldots.
\end{equation*}
The naturality of the long exact sequence guarantees that the following diagram is commutative:
\[\begin{tikzcd}
	\ldots & {H_c^m(D^{(n)}-D^{(n-1)})} & {H_c^m (D)} & {H_c^m(D^{(n-1)})} & \ldots \\
	\ldots & {H_c^m(D^{(n)}-D^{(n-1)})} & {H_c^m (D)} & {H_c^m(D^{(n-1)})} & \ldots
	\arrow[from=1-1, to=1-2]
	\arrow[from=1-2, to=1-3]
	\arrow["{f^\ast}"', from=1-2, to=2-2]
	\arrow[from=1-3, to=1-4]
	\arrow["{f^\ast}"', from=1-3, to=2-3]
	\arrow[from=1-4, to=1-5]
	\arrow["{f^\ast}"', from=1-4, to=2-4]
	\arrow[from=2-1, to=2-2]
	\arrow[from=2-2, to=2-3]
	\arrow[from=2-3, to=2-4]
	\arrow[from=2-4, to=2-5],
\end{tikzcd}\]
we obtain
\begin{equation*}
    \mathrm{L}_c(D,f)=\mathrm{L}_c(D^{(n)}-D^{(n-1)},f)+\mathrm{L}_c(D^{(n-1)},f).
\end{equation*}
But $D^{(n-1)}$ is a locally compact space (it is a closed subspace of a locally compact space) so we can repeat the argument and obtain
\begin{equation*}
    \mathrm{L}_c(D^{(n-1)},f)=\mathrm{L}_c(D^{(n-1)}-D^{(n-2)},f)+\mathrm{L}_c(D^{(n-2)},f)
\end{equation*}
Finally, as $D^{(0)}=D^{(0)}-D^{(-1)}$, we arrive to the desired result.
\end{proof}

\begin{theorem}[Equivalence of Lefschetz numbers]\label{iso comp comb}
    Let $X$ be a definable cellular complex and let $D$ be a definable locally compact incomplete subcomplex. If $f\colon X\to X$ is a cellular homeomorphism such that $D$ is $f$-invariant, then $\varLambda_c(D,f)=\mathrm{L}_c(D,f)$.
\end{theorem}
\begin{proof}
    For each dimension $i$, as $f$ is a cellular homeomorphism and $D$ is $f$-invariant, we have that $X^{(i)}$ , $D^{(i)}-D^{(i-1)}$ and $X^{(i)}-(D^{(i)}-D^{(i-1)})$ are $f$-invariant. As each $X^{(i)}$ is a definable cellular complex, by additivity of the combinatorial Lefschetz number (Lemma \ref{lemma:additivity_combinatorial_lefschetz_number}), we have
    \begin{equation}\label{ec 1 teor comp comb}
        \varLambda_c(X^{(i)},f)=\varLambda_c(D^{(i)}-D^{(i-1)},f)+\varLambda_c(X^{(i)}-(D^{(i)}-D^{(i-1)}),f).
    \end{equation}
    Furthermore, by an argument similar to the one of Theorem~\ref{adit sop comp}, we have:
    \begin{equation}\label{ec 2 teor comp comb}
        \mathrm{L}_c(X^{(i)},f)=\mathrm{L}_c(D^{(i)}-D^{(i-1)},f)+\mathrm{L}_c(X^{(i)}-(D^{(i)}-D^{(i-1)}),f).
    \end{equation}
    Now, note that the first and the last term in Equation~\ref{ec 1 teor comp comb} equal the Lefschetz homological (or cohomological as we work on a field) number since the spaces are triangulated by complete (compact) simplicial complexes.  The first and the last terms in Equation~\ref{ec 2 teor comp comb} are also the Lefschetz cohomological number because  of Lemma~\ref{lemas previos}~(D) and because the spaces are compact. As a consequence, we obtain that $\varLambda_c(D^{(i)}-D^{(i-1)},f)=\mathrm{L}_c(D^{(i)}-D^{(i-1)},f)$.
    Finally, from Theorem \ref{adit sop comp}:
    \begin{equation*}
        \mathrm{L}_c(D,f)=\sum_i \mathrm{L}_c(D^{(i)}-D^{(i-1)},f)
    \end{equation*}
    and
    \begin{equation*}
        \varLambda_c(D,f)=\sum_i \varLambda_c(D^{(i)}-D^{(i-1)},f)
    \end{equation*}
    we obtain the result.
\end{proof}

\begin{corollary}[Well-definedness of the combinatorial Lefschetz number]\label{coro:well-definedness_combinatorial_number}
Let $X$ be a definable cellular complex and let $D$ be a definable locally compact incomplete subcomplex. Let $f\colon X\to X$ be a cellular homeomorphism such that $D$ is $f$-invariant.  
Let $(Y,K,g)$ and $(Y',K',g')$ be two triangulations of $(X,D,f)$. Then $\varLambda_c(K,g)=\varLambda_c(K',g')$.
\end{corollary}

\begin{proof}
    By Theorem \ref{iso comp comb} it is enough to prove that $L_c(K,g)=L_c(K',g')$. Consider the commutative diagram (where we abuse of notation for simplicity):
$$\begin{tikzcd}
K \arrow[r, "g"]                                 & K                                 \\
D \arrow[u, "h"] \arrow[d, "h'"'] \arrow[r, "f"] & D \arrow[u, "h"'] \arrow[d, "h'"] \\
K' \arrow[r, "g'"]                               & K'                               
\end{tikzcd}$$
where $h\colon D\to K$ and $h'\colon D\to K'$ are the triangulation homeomorphisms. We obtain the following commutative diagram induced in cohomology with compact supports:
\[\begin{tikzcd}
	{H_c^\ast(K)} & {H_c^\ast(K)} \\
	{H_c^\ast(K')} & {H_c^\ast(K')}
	\arrow["{g^{\ast}}", from=1-1, to=1-2]
	\arrow["{(h(h')^{-1})^\ast}", from=1-1, to=2-1]
	\arrow["{g'^\ast}", from=2-1, to=2-2]
	\arrow["{(h'(h)^{-1})^\ast}"', from=2-2, to=1-2]
\end{tikzcd}\]
where $(h'(h)^{-1})^\ast$ and $(h(h')^{-1})^\ast$ are isomorphisms. By linear algebra, $L_c(K,g)=L_c(K',h{h'}^{-1}g'h'h^{-1}).$ 
As a consequence, 
$L_c(K,g)=L_c(K',g')$.
\end{proof}
\begin{corollary}[Topological invariance of the combinatorial Lefschetz number]\label{coro:invariance_comb_lefschetz}
    Let be $X$ and $Y$ definable cellular complexes and let $C\subset X$ and $D\subset Y$ definable locally compact incomplete subcomplexes. Let $f:X\rightarrow X$ and $g:Y\rightarrow Y$ be cellular homeomorphisms such that $C$ (resp. $D$) is $f$-invariant (resp. $g$-invariant). If there is a homeomorphism $h:X\rightarrow Y$ such that the following diagram commutes
    \[\begin{tikzcd}
	U & U \\
	V & V,
	\arrow["f", from=1-1, to=1-2]
	\arrow["h"', from=1-1, to=2-1]
	\arrow["h", from=1-2, to=2-2]
	\arrow["g", from=2-1, to=2-2]
\end{tikzcd}\]
then we have $\varLambda_c(C,f) = \varLambda_c(D,g)$.
\end{corollary}

\section{The sheaf-theoretic Lefschetz number as an integral}
In this section we obtain the Representation Theorem for $f$-$c$-constructible sheaves (Theorem \ref{teor integral}) which evidences the strong relation between the sheaf-theoretic Lefschetz number ad integration with respect to the (combinatorial) Lefschetz number.

\subsection{$\mathbb{N}$-constructible functions and constructible sheaves}\label{subsec:integration_comb_lefschetz_number}
Let $f:X\rightarrow X$ be a cellular  homeomorphism of a definable cellular complex to itself. A function $h:X\to \mathbb{N}$ ($\mathbb{N}=\{0,1,\ldots\}$) is $f$-\textit{integrable} or $f$-\textit{constructible} if there are a finite number of level sets $\{X_j\}_{j=0}^n=\{h^{-1}(j)\}_{j=0}^{n}$ and each level set is a locally compact, definable, $f$-invariant  incomplete subcomplex. Note that $h=\sum _{j=1}^n j\mathds{1}_{h^{-1}(j)}$ (where $\mathds{1}_{h^{-1}(j)}$ stands for the characteristic function on $h^{-1}(j)$) but there may it be other expressions for $h$ of the form $h=\sum _{j=1}^n c_j\mathds{1}_{U_j}$, where the $c_j\in \mathbb{N}$ and the sets $U_j$ are definable and $f$-invariant. We define 
\begin{equation*}
\int_{X}(c_1\mathds{1}_{U_1}+ \ldots +c_n\mathds{1}_{U_n})\,d\varLambda_c f=c_1\varLambda_c(f,U_1)+ \ldots +c_n\varLambda_c(f,U_n).
\end{equation*}
This is well defined (\cite[Theorem 5.2]{M-M}). For an example of a constructible function and an integral, see \cite[Example 7.3]{M-M}.

A $f$-$c$-constructible sheaf $\mathcal{F}$ in $X$ defines the $f$-constructible function $h:X\to \mathbb{N}$ that sends each point $x\in X$ to the dimension of the fiber $\mathcal{F}_x$ and where the level sets consist of locally compact incomplete subcomplexes. Conversely, given a $f$-constructible function $h:X\to \mathbb{N}$ in the cellular complex $X$, we can define a $f$-$c$-constructible sheaf that encodes $h$. To do so, in each $X_j$ we define the constant sheaf $\tilde{\mathbb{R}}_{X_j}^j$ as the extension by zero to $X$ of the constant sheaf on $X_j$ (such extension exists because $X_j$ is locally closed). We define the $f$-$c$-constructible sheaf $\mathcal{F}$ \textit{associated to h} in $X$ given by $\mathcal{F}=\oplus_j \tilde{\mathbb{R}}^j_{X_j}$. 

\begin{proposition}\label{teor coho haz cte}
    Let $X$ be  a cellular complex and $U\subset X$ a definable locally compact incomplete subcomplex. Let $f:X\rightarrow X$ be a cellular homeomorphism such that $U$ is $f$-invariant and let  $\mathcal{F}$ be the constant sheaf of dimension $n$. Then the sheaf-theoretic Lefschetz number of $f$ associated to $\mathcal{F}$ on $U$ is equal to $n$ times the combinatorial Lefschetz number of $f$ on $U$, that is
    \begin{equation*}
L_c(U,f,\mathcal{F})=n\varLambda_c(U,f).
    \end{equation*}
\end{proposition}
\begin{proof}
    First of all, as we deal with finite dimensional vector spaces, we must recall that the sheaf $\mathcal{F}$ is nothing more the $n$-sum of the constant sheaf with value $\mathbb{R}$. So, if 
    \begin{equation*}
        0\rightarrow \mathcal{R}\rightarrow\mathcal{J}^0\rightarrow\mathcal{J}^1\rightarrow\mathcal{J}^2\rightarrow \ldots
    \end{equation*}
    is an inyective resolution of the constant sheaf with value $\mathbb{R}$, then 
    \begin{equation*}
        0\rightarrow\mathcal{F}\rightarrow\oplus_{n}\mathcal{J}^0\rightarrow\oplus_{n}\mathcal{J}^1\rightarrow\oplus_{n}\mathcal{J}^2\rightarrow\ldots
    \end{equation*}
    is an inyective resolution of $\mathcal{F}$.

    By taking sections with compact supports and removing the first term, we have
    \begin{equation*}
        0\rightarrow \oplus_{n}\Gamma_c(U,\mathcal{J}^0)\rightarrow \oplus_{n}\Gamma_c(U,\mathcal{J}^1)\rightarrow \oplus_{n}\Gamma_c(U,\mathcal{J}^2)\rightarrow\ldots
    \end{equation*}
    and so $ H_c^{\ast}(U,\mathcal{F})\cong \oplus_n H_c^{\ast}(U,\mathbb{R})$. Then, by Lemma~\ref{coho cte}, we conclude that $ H_c^{\ast}(U,\mathcal{F})\cong\oplus_n H_c^{\ast}(U;\mathbb{R})$. As a consequence, we have that $\varLambda_c(U,f,\mathcal{F})=n\mathrm{L}_c(U,f)$, and Theorem~\ref{iso comp comb} guarantees that the last term is $n\varLambda_c(U,f)$.
\end{proof}

\begin{lemma}
    We have $H^m_c(X,\mathcal{F})=0$ for a sufficiently large $m$.
\end{lemma}

\begin{proof}
    Let $U_{1,n}$ be the open set consisting of the top dimensional cells that are in a constant component $X_1$ of the sheaf. As it is open and $f$-invariant, we can consider the following diagram
    \[\begin{tikzcd}
	{} & {H^m_c(U_{1,n},\mathcal{F})} & {H^m_c(X,\mathcal{F})} & {H^m_c(X-U_{1,n},\mathcal{F})} & {} \\
	{} & {H^m_c(U_{1,n},\mathcal{F})} & {H^m_c(X,\mathcal{F})} & {H^m_c(X-U_{1,n},\mathcal{F})} & {}
	\arrow[from=1-1, to=1-2]
	\arrow[from=1-2, to=1-3]
	\arrow["{f^{\ast}}"', from=1-2, to=2-2]
	\arrow[from=1-3, to=1-4]
	\arrow["{f^{\ast}}"', from=1-3, to=2-3]
	\arrow[from=1-4, to=1-5]
	\arrow["{f^{\ast}}"', from=1-4, to=2-4]
	\arrow[from=2-1, to=2-2]
	\arrow[from=2-2, to=2-3]
	\arrow[from=2-3, to=2-4]
	\arrow[from=2-4, to=2-5].
\end{tikzcd}\]
    As, by the proof of Proposition~\ref{teor coho haz cte} $H^n_c(U_{1,n},\mathcal{F})\cong \oplus_s H^n_c(U_{1,n}; \mathbb{R})$ for certain integer $s$, we have that this term vanishes for a sufficiently large $m$. Now, if we repeat the process with $U_{2,n}$, we obtain

    \[\begin{tikzcd}[column sep=tiny]
	{} & {H^m_c(U_{2,n},\mathcal{F})} & {H^m_c(X-U_{1,n},\mathcal{F})} & {H^m_c(X-\{U_{1,n}\cup U_{2,n}\},\mathcal{F})} & {} \\
	{} & {H^m_c(U_{2,n},\mathcal{F})} & {H^m_c(X-U_{1,n},\mathcal{F})} & {H^m_c(X-\{U_{1,n}\cup U_{2,n}\},\mathcal{F})} & {}
	\arrow[from=1-1, to=1-2]
	\arrow[from=1-2, to=1-3]
	\arrow["{f^{\ast}}"', from=1-2, to=2-2]
	\arrow[from=1-3, to=1-4]
	\arrow["{f^{\ast}}"', from=1-3, to=2-3]
	\arrow[from=1-4, to=1-5]
	\arrow["{f^{\ast}}"', from=1-4, to=2-4]
	\arrow[from=2-1, to=2-2]
	\arrow[from=2-2, to=2-3]
	\arrow[from=2-3, to=2-4]
	\arrow[from=2-4, to=2-5].
\end{tikzcd}\]

Once again, the term $H^m_c(U_{2,n},\mathcal{F})$ vanishes for a sufficiently large $m$. By repeating this process finitely many times the right term of the diagram vanishes for a sufficiently large $m$. Then, if we climb on the middle term of the diagrams, we see that $H^m_c(X,\mathcal{F})$ must vanish for a sufficiently large $m$. 
\end{proof}

\subsection{Sheaf-theoretic Lefschetz number and Lefschetz calculus}
We establish a relation between the Lefschetz number on sheaves and the integral with respect to the Lefschetz number.

\begin{theorem}[Representation theorem]\label{teor integral}
    Let $X$ be a definable cellular complex and let $f\colon X \to X$ be a cellular homeomorphism. If $\mathcal{F}$ is a $f$-$c$-constructible sheaf on $X$ such that $f\colon X \to X$ respects the constant components of $\mathcal{F}$, then 
\begin{equation}\label{eq:corresopndence_sheaves_integrals_functions}
        L_c(X,f,\mathcal{F})=\int_X h\mathrm{d}\varLambda f,
    \end{equation}
    where $h\colon X \to \mathbb{N}$ is the $f$-constructible function associated to $\mathcal{F}$.
    Conversely, if $h\colon X \to \mathbb{N}$ is a $f$-constructible function $h=\sum _{j=1}^n c_j\mathds{1}_{U_j}$, its associated $f$-$c$-constructible sheaf $\mathcal{F}$ satisfies Equation (\ref{eq:corresopndence_sheaves_integrals_functions}).
\end{theorem}

\begin{proof}
    Let us begin by considering $U_{1,n}$ as the open set consisting of the top dimensional cells that are in a constant component $X_1$ of the sheaf. As it is open and $f$-invariant, we can consider the following diagram
\[\begin{tikzcd}
	{} & {H^n_c(U_{1,n},\mathcal{F})} & {H^n_c(X,\mathcal{F})} & {H^n_c(X-U_{1,n},\mathcal{F})} & {} \\
	{} & {H^n_c(U_{1,n},\mathcal{F})} & {H^n_c(X,\mathcal{F})} & {H^n_c(X-U_{1,n},\mathcal{F})} & {}
	\arrow[from=1-1, to=1-2]
	\arrow[from=1-2, to=1-3]
	\arrow["{f^{\ast}}"', from=1-2, to=2-2]
	\arrow[from=1-3, to=1-4]
	\arrow["{f^{\ast}}"', from=1-3, to=2-3]
	\arrow[from=1-4, to=1-5]
	\arrow["{f^{\ast}}"', from=1-4, to=2-4]
	\arrow[from=2-1, to=2-2]
	\arrow[from=2-2, to=2-3]
	\arrow[from=2-3, to=2-4]
	\arrow[from=2-4, to=2-5].
\end{tikzcd}\]
     As a consequence, since the three cohomology groups vanish on a sufficiently large dimension $n$, we have
    \begin{equation*}
\varLambda_c(X,f,\mathcal{F})=\varLambda_c(U_{1,n},f,\mathcal{F})+\varLambda_c(X-U_{1,n},\mathcal{F}).
    \end{equation*}
    Now, if we repeat the argument taking $U_{2,n}$ instead of $U_{1,n}$ and $X-U_{1,n}$ instead of $X$ we arrive to 
    \begin{equation*}
\varLambda_c(X,f,\mathcal{F})=\varLambda_c(U_{1,n},f,\mathcal{F})+\varLambda_c(U_{2,n},\mathcal{F})+\varLambda_c(X-U_{1,n}-U_{2,n},\mathcal{F}).
    \end{equation*}
    Then we can repeat this argument until arrive to
    \begin{equation*}
        \varLambda_c(X,f,\mathcal{F})=\sum_{j=1}^{m}\varLambda_c(U_{j,n},f,\mathcal{F})+\varLambda_c(X^{(n-1)},f,\mathcal{F}),
    \end{equation*}
    and then apply the same for the cells in progresive lower dimension. 

    Finally we obtain
    \begin{equation}\label{ec aditividad}
        \varLambda_c(X,f,\mathcal{F})=\sum_{i=0}^n\left(  \sum_{j=1}^m \varLambda_c(U_{j,i},f,\mathcal{F})  \right)
    \end{equation}

We apply Theorem~\ref{teor coho haz cte} to each summand of the right-hand side of the equation. Then,
\begin{equation*}
    \varLambda_c(X,f,\mathcal{F})=\sum_{i=0}^n  \left(\sum_{j=i}^m c_j \varLambda_c(U_{j,i},f)\right),
\end{equation*}
where $c_j$ denotes the dimension of the sheaf on $X_j$.  Regrouping the sum and using the the additivity of the combinatorial Lefschetz number we obtain:
\begin{equation*}
    \sum_{i=0}^n  \left(\sum_{j=i}^m c_j\varLambda_c(U_{j,i},f)\right)= \sum_{j=i}^m c_j \left(\sum_{i=0}^n \varLambda_c(U_{j,i},f)\right)=\sum_{j=i}^m c_j \varLambda_c(X_{j},f)
\end{equation*}
where $X_{j}=\sqcup_iU_{j,i}$. Observe that $\int_x h\mathrm{d}\varLambda f=\sum_{j=i}^m c_j \varLambda_c(X_{j},f)$ where $h=\sum _{j=1}^n c_j\mathds{1}_{X_j}$ is the associated $f$-constructible function to the $f$-$c$-constructible sheaf $\mathcal{F}$. For the converse of the theorem, note that the associated associated $f$-constructible function to the associated $f$-$c$-constructible sheaf $\mathcal{F}$ to $h\colon X\to \mathbb{N}$ is $h\colon X\to \mathbb{N}$.
\end{proof}
The following example illustrates  Theorem~\ref{teor integral}.
\begin{example}
    Let us consider as cellular complex $X$ the simplicial complex of Figure~\ref{figura complejo}. 
         \begin{figure}[htb!]
    \centering
     \includegraphics[scale=0.50]{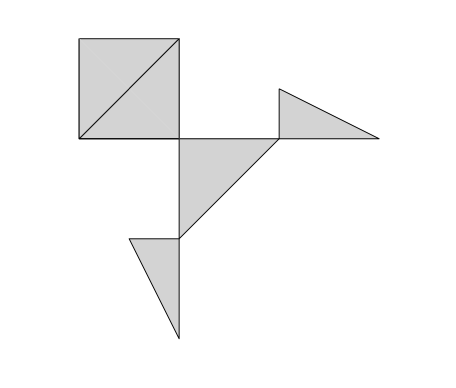} 
     \caption{Complex X.}
    \label{figura complejo}
\end{figure}
    
For simplicity, we take as the sheaf $\mathcal{F}$ a sheaf given with the structure of an associated sheaf. Consider the partition $\{X_1, X_2, X_3, X_4, X_5\}$ given in Figure~\ref{figura particion}.
  \begin{figure}[htb!]
    \centering
     \includegraphics[scale=0.50]{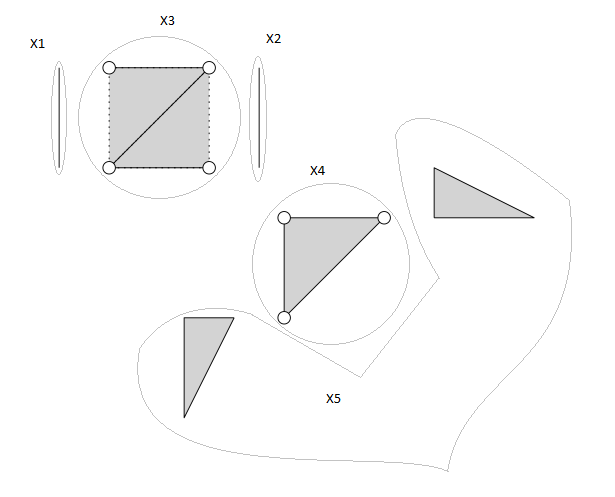} 
     \caption{Elements of the partition.}
    \label{figura particion}
\end{figure}

We take the sheaf $\mathcal{F}=i_1!\mathcal{F}_1\oplus i_2!\mathcal{F}_2\oplus i_3!\mathcal{F}_3 \oplus i_4!\mathcal{F}_4 \oplus i_5!\mathcal{F}_5$, where $\mathcal{F}_1=\tilde{\mathbb{R}}^1,\;\mathcal{F}_2=\tilde{\mathbb{R}}^4,\;\mathcal{F}_3=\tilde{\mathbb{R}}^2,\;\mathcal{F}_4=\tilde{\mathbb{R}}^1$ and $\mathcal{F}_5=\tilde{\mathbb{R}}^3$ are constant sheaves.

Let us present now the map $f$. If we see the upper square as $[0,1]\times[0,1]$, we take $f$ there to be the homeomorphism that sends each $(x,y)$ to $(x+x(1-x)(x-y),y)$. This map enlarges the intermediate parts of the square and fixes the edges and the diagonal. If we take the intermediate triangle as the convex hull of the points $(-1,0), (1,0)$ and $(0,1)$ of the plane, then we define $f$ there to be the homeomorphism that sends $(x,y)$ to $(-x-x(y-(1-x)),y)$ if $x\le 0$ and to $((-x-x(y-(1+x))),y)$ if $x>0$. This map reflexes the triangle and contracts the interior and the base, fixing the vertical line at $(0,0)$. Finally, the map $f$ is taken to swap the triangles of $X_5$. Note that here $\mathcal{F}=\mathcal{F}'$ and the $X_i's$ are precisely those shown  in Figure~\ref{figura particion}.

Now we'd like to compute the Lefschetz number $\varLambda_c(X,f,\mathcal{F})$ without computing directly the cohomology of the sheaf. We will use Theorem~\ref{teor integral}. The first step is to consider the associated constructible funcion $h$ defined by
\begin{equation*}
    h= 1\cdot\mathds{1}_{X_1}+4\cdot\mathds{1}_{X_2}+2\cdot\mathds{1}_{X_3}+1\cdot\mathds{1}_{X_4}+3\cdot\mathds{1}_{X_5}.
\end{equation*}
Now we compute
\begin{equation*}
    \int_X h d\varLambda f= 1\varLambda_c(X_1,f)+4\varLambda_c(X_2, f)+ 2\varLambda_c(X_3, f) + 1\varLambda_c(X_4, f) + 3\varLambda_c(X_5,f).
\end{equation*}
Now, as the identity is a simplicial approximation of $f$ at the square and the reflection is a simplicial approximation of $f$ at the middle triangle, this leads to
\begin{equation*}
    L_c(X,f,\mathcal{F})=\int_X hd\varLambda f =1+4-2+0+0=3.
\end{equation*}

\end{example}



\subsection{$\mathbb{Z}$-valuated constructible functions and Barrow's rule.}

We study constructible functions valued in the integers ($\mathbb{Z}$) by considering complexes of sheaves instead of sheaves.

In the notation of \cite{C-G-M}, if we consider a bounded complex of $f$-constructible sheaves, that is:
\begin{equation*}
    \mathcal{F}^\bullet= \ldots \xrightarrow{d} \mathcal{F}^{i-1} \xrightarrow{d} \mathcal{F}^{i} \xrightarrow{d} \mathcal{F}^{i+1}\xrightarrow{d} \ldots,
\end{equation*}
then we obtain an associated $\mathbb{Z}$ valued $f$-constructible function by
\begin{equation*}
    h(x)=\sum_i (-1)^i\mathrm{dim}(\mathcal{H}^i\mathcal{F}_x^i),
\end{equation*}
where the level sets consist of locally compact incomplete subcomplexes.
Conversely, given one of these $f$-constructible functions with values on $\mathbb{Z}$, then there exists a complex of $f$-$c$-constructible sheaves $\mathcal{F}^\bullet$ such that $h(x)=\sum_i (-1)^i \mathrm{dim}(\mathcal{H}^i, \mathcal{F}_x^\bullet)$. For it suffices to consider the complex
     \[\begin{tikzcd}
	0 & {\oplus_{n<0}\tilde{\mathbb{R}}_{X_n}^{-n}} & {\oplus_{n>0}\tilde{\mathbb{R}}_{X_n}^{n}} & 0
	\arrow["0", from=1-1, to=1-2]
	\arrow["0", from=1-2, to=1-3]
	\arrow["0", from=1-3, to=1-4],
\end{tikzcd}\]  
where the $X_j$'s are the level sets of $h$ ($X_j=h^{-1}(j)$).
Again, the relationship in our case will be between bounded complexes of  $c$-constructible sheaves on $X$ and cellular $f$-constructible functions.

\begin{theorem}
    [Barrow's rule]\label{coro:barrow}
    Let be $h$ a  $f$-constructible function on $X$ with level sets consisting of locally compact incomplete subcomplexes and let be $\{X_j\}$ the level sets of $h$. Then 
    \begin{equation*}
        \int_X h \mathrm{d}\varLambda f= L_c(X,f,\oplus_{j>0}\tilde{\mathbb{R}}^j_{X_j})-L_c(X,f,\oplus_{j<0}\tilde{\mathbb{R}}^{-j}_{X_j}).
    \end{equation*}
\end{theorem}

\begin{proof}
    By the definition of the integral we have
    \begin{align*}
        \int_X hd\varLambda f&= \sum_{j\in \mathbb{Z}} j\cdot\varLambda_c(X_j,f)=\sum_{j<0} j\cdot\varLambda_c(X_j,f)+\sum_{j>0} j\cdot\varLambda_c(X_j,f)\\
        &=\sum_{j>0} j\cdot\varLambda_c(X_j,f)-\sum_{j>0} j\cdot\varLambda_c(X_{-j},f)\\
        &=L_c(X,f,\oplus_{j>0}\tilde{\mathbb{R}}^j_{X_j})-L_c(X,f,\oplus_{j<0}\tilde{\mathbb{R}}^{-j}_{X_j}). \qedhere
    \end{align*}
\end{proof}

\section{Axiomatization of the sheaf-theoretic Lefschetz number}\label{sec:axiomatization_lefschetz_number}

In this section, we study to what extent the Representation Theorem for $f$-$c$-constructible sheaves (Theorem \ref{teor integral}) characterizes and determines the sheaf-theoretic Lefschetz number. We begin by introducing the family where we are going to define the axioms.
\begin{definition}
    Let be $\mathcal{C}$ the family of triples $(X,f,\mathcal{F})$ where $X$ is a connected definable cellular complex, $\mathcal{F}$ is a $f$-$c$-constructible sheaf and $f$ a cellular homeomorphism that respects the constant components of $\mathcal{F}$.
\end{definition}

\begin{theorem}\label{thm:axiomatization_lefschetz}
    The sheaf-theoretic Lefschetz number is the only map between $\mathcal{C}$ and $\mathbb{Z}$ that satisfies the following axioms:
    \begin{enumerate}
        \item (Cofibration Axiom). Let $X$ be a definable CW-complex and $f\colon X\to X$ a cellular homeomorphism. If $A$ is a connected $f$-invariant definable subcomplex, $A\rightarrow X\rightarrow X/A$ is the resulting cofiber sequence and there exist a commutative diagram 
        \[\begin{tikzcd}
	A & X & X/A \\
	A & X & X/A
	\arrow["{f'}"', from=1-1, to=2-1]
	\arrow["f"', from=1-2, to=2-2]
	\arrow["{\overline{f}}"', from=1-3, to=2-3]
	\arrow[from=1-1, to=1-2]
	\arrow[from=1-2, to=1-3]
	\arrow[from=2-1, to=2-2]
	\arrow[from=2-2, to=2-3]
\end{tikzcd}\]
then 
\begin{equation*}
    L_c(X,f,\mathcal{F})=L_c(X/A, \overline{f},\pi_{\ast}\mathcal{F})+L_c(A,f',\mathcal{F})-r,
\end{equation*}
where $r$ is the dimension of the sheaf $(\pi_\ast\mathcal{F)}_{|A/A}$
\item (Wedge Spheres Axiom) If $X$ is homeomorphic to a finite wedge of $n$-spheres or a complex of dimension $0$ or $1$, $\mathcal{F}$ is a $f$-$c$-constructible sheaf and $f$ is a cellular homeomorphism that respects the constant components of $\mathcal{F}$, then
\begin{equation*}
    L_c(X,f,\mathcal{F})=\int_X h d\varLambda f,
\end{equation*}
where $h$ is the constructible function associated to the sheaf.
    \end{enumerate}
\end{theorem}
\begin{proof}
Let us start by cheeking that the Lefschetz number satisfies the axioms.
    First at all, since $A$ is a definable cellular subcomplex of $X$ (in particular closed, locally compact and with definable skeletons on all dimensions) and $f$-invariant, $\mathcal{F}$ will be also a $c$-constructible sheaf on $A$. Moreover, 
        since $X-A$ is also locally compact (it is open in a locally compact space) and $f$-invariant, it makes sense to consider the following diagram with exact rows, where vertical arrows are induced by $f$.
        \[\begin{tikzcd}[row sep=normal,column sep=tiny]
	0 & {H_c^0(X- A,\mathcal{F})} & {H_c^0(X,\mathcal{F})} & {H_c^0(A,\mathcal{F})} & {H_c^1(X- A,\mathcal{F})} & \ldots \\
	0 & {H_c^0(X- A,\mathcal{F})} & {H_c^0(X,\mathcal{F})} & {H_c^0(A,\mathcal{F})} & {H_c^1(X- A,\mathcal{F})} & \ldots
	\arrow[from=1-1, to=1-2]
	\arrow[from=1-2, to=1-3]
	\arrow[from=1-3, to=1-4]
	\arrow[from=1-4, to=1-5]
	\arrow[from=1-5, to=1-6]
	\arrow[from=2-1, to=2-2]
	\arrow[from=2-2, to=2-3]
	\arrow[from=2-3, to=2-4]
	\arrow[from=2-4, to=2-5]
	\arrow[from=2-5, to=2-6]
	\arrow[from=1-2, to=2-2]
	\arrow[from=1-3, to=2-3]
	\arrow[from=1-4, to=2-4]
	\arrow[from=1-5, to=2-5]
\end{tikzcd}\]
In consequence, since the three cohomology groups vanish on a sufficiently large dimension $n$, we have
\begin{align}
    \label{ec cofibr} \sum_{i} (-1)^i \mathrm{tr}(f^{\ast},H_c^i(X,\mathcal{F}))&=\sum_i (-1)^i \mathrm{tr}(f^{\ast},H_c^i(X-A,\mathcal{F}))\\
    \nonumber &+\sum_i (-1)^i \mathrm{tr}(f^{\ast},H_c^i(A,\mathcal{F})).
\end{align}
Now, if we take the quotient map $\pi :X\rightarrow X\backslash A$ we can consider the direct image of $\mathcal{F}$ by $\pi$. As $A$ is definable, by \cite[Chapter 10, 2.4]{Dries} we have that the quotient map and the quotient space by $A$ can be assumed to be definable. As $A/A$ is definable (it is a single point) then $X/A-A/A$ will be definable. Now, as $(X/A-A/A, \pi_\ast \mathcal{F})$ is isomorphic to $(X-A,\mathcal{F})$, we have that there is a definable partition of the cellular complex $X/A$ consisted on the one of $X-A$ and the point $A/A$ such that, on each element of the partition, the sheaf is contant on its connected components. So we can check that $\pi_\ast\mathcal{F}$ is a $c$-constructible sheaf on $X/A$ (and so also on $X/A-A/A$ and on $A/A$). Furthermore, as $A/A$ is closed in $X/A$ we have the following sequence
\[\begin{tikzcd}[row sep=small,column sep=normal]
	0 & 0 \\
	0 & 0 \\
	{H_c^0(X/A- A/A,\pi_{\ast}\mathcal{F})} & {H_c^0(X/A-  A/A,\pi_{\ast}\mathcal{F})} \\
	{H_c^0(X/A, \pi_{\ast}\mathcal{F})} & {H_c^0(X/A, \pi_{\ast}\mathcal{F})} \\
	{H_c^0(A/A, \pi_{\ast}\mathcal{F})} & {H_c^0(A/A, \pi_{\ast}\mathcal{F})} \\
	{H_c^1(X/A- A/A,\pi_{\ast}\mathcal{F})} & {H_c^1(X/A- A/A,\pi_{\ast}\mathcal{F})} \\
	{H_c^1(X/A, \pi_{\ast}\mathcal{F})} & {H_c^1(X/A, \pi_{\ast}\mathcal{F})} \\
	0 & 0 \\
	{H_c^2(X/A- A/A,\pi_{\ast}\mathcal{F})} & {H_c^2(X/A- A/A,\pi_{\ast}\mathcal{F})} \\
	{H_c^2(X/A, \pi_{\ast}\mathcal{F})} & {H_c^2(X/A,\pi_{\ast}\mathcal{F})} \\
	0 & 0 \\
	\ldots & \ldots.
	\arrow[from=1-1, to=2-1]
	\arrow[from=1-2, to=2-2]
	\arrow[from=2-1, to=3-1]
	\arrow[from=2-2, to=3-2]
	\arrow[from=3-1, to=4-1]
	\arrow[from=3-2, to=4-2]
	\arrow[from=4-2, to=5-2]
	\arrow[from=4-1, to=5-1]
	\arrow[from=5-1, to=6-1]
	\arrow[from=5-2, to=6-2]
	\arrow[from=6-1, to=6-2]
	\arrow[from=5-1, to=5-2]
	\arrow[from=4-1, to=4-2]
	\arrow[from=3-1, to=3-2]
	\arrow[from=2-1, to=2-2]
	\arrow[from=11-1, to=12-1]
	\arrow[from=11-2, to=12-2]
	\arrow[from=11-1, to=11-2]
	\arrow[from=10-2, to=11-2]
	\arrow[from=10-1, to=11-1]
	\arrow[from=10-1, to=10-2]
	\arrow[from=9-1, to=10-1]
	\arrow[from=8-2, to=9-2]
	\arrow[from=9-2, to=10-2]
	\arrow[from=9-1, to=9-2]
	\arrow[from=8-1, to=9-1]
	\arrow[from=8-1, to=8-2]
	\arrow[from=6-2, to=7-2]
	\arrow[from=7-2, to=8-2]
	\arrow[from=6-1, to=7-1]
	\arrow[from=7-1, to=8-1]
	\arrow[from=7-1, to=7-2]
\end{tikzcd}\]
From here we see that, for $n>1$, 
\begin{equation}\label{ec cofibr 2}
H^n_c(X/A-A/A,\pi_\ast\mathcal{F})\simeq H^n_c(X/A,\pi_\ast\mathcal{F}).   
\end{equation}
Also by considering the first seven rows we obtain
\begin{align*}
    & \mathrm{tr}(f^{\ast},H^0_c(X/A-A/A,\pi_\ast\mathcal{F}))-\mathrm{tr}(f^{\ast},H^1_c(X/A-A/A,\pi_\ast\mathcal{F}))\\
    &=\mathrm{tr}(f^{\ast},H^0_c(X/A,\pi_\ast\mathcal{F}))-\mathrm{tr}(f^{\ast},H^1_c(X/A,\pi_\ast\mathcal{F}))-\mathrm{tr}(f^{\ast},H_c^0(A/A,\pi_\ast\mathcal{F})).
\end{align*}
So, from here and from Equation~\ref{ec cofibr 2} if we replace terms in Equation~\ref{ec cofibr} we obtain 
\begin{align*}
    &\sum_{i} (-1)^i \mathrm{tr}(f^{\ast},H_c^i(X,\mathcal{F}))\\
    =&\sum_i (-1)^i \mathrm{tr}(f^{\ast},H_c^i(X-A,\mathcal{F}))+\sum_i (-1)^i \mathrm{tr}(f^{\ast},H_c^i(A,\mathcal{F}))\\
    =&\sum_{i>1}(-1)^i\mathrm{tr}(f^{\ast},H_c^i(X/A,\pi_\ast\mathcal{F}))+\mathrm{tr}(f^{\ast},H_c^0(X-A,\mathcal{F}))\\&-\mathrm{tr}(f^{\ast},H_c^1(X-A,\mathcal{F})) -\mathrm{tr}(f^{\ast},H_c^0(A/A,\pi_\ast\mathcal{F})) +\sum_i (-1)^i \mathrm{tr}(f^{\ast},H_c^i(A,\mathcal{F}))\\
    =&\sum_{i}(-1)^i\mathrm{tr}(f^{\ast},H_c^i(X/A,\pi_\ast\mathcal{F}))-\mathrm{tr}(f^{\ast},H_c^0(A/A,\pi_\ast\mathcal{F}))+\sum_i (-1)^i \mathrm{tr}(f^{\ast},H_c^i(A,\mathcal{F}))\\
    =&\sum_{i}(-1)^i\mathrm{tr}(f^{\ast},H_c^i(X/A,\pi_\ast\mathcal{F}))-\mathrm{dim}((\pi_\ast\mathcal{F)}_{|A/A})+\sum_i (-1)^i \mathrm{tr}(f^{\ast},H_c^i(A,\mathcal{F})).
\end{align*}

The Lefschetz number satisfies the Wedge Spheres Axiom because of Theorem~\ref{teor integral}.

    Now we must cheek that the Lefschetz number is the only map between $\mathcal{C}$ and $\mathbb{Z}$ that satisfies this two axioms.

    We proceed by induction on the dimension of $X$. If $X$ is of dimension $1$ it is due to the Wedge Spheres Axiom. Furthermore, if $X$ is a wedge product of spheres, it is also a consequence of the Wedge Spheres Axiom. In other case, let $(X,f,\mathcal{F})$ be an element of $\mathcal{C}$, with $X$ a cellular complex of dimension $n$. Then, we can consider the diagram
    \[\begin{tikzcd}
	{X^{n-1}} & X & {X/X^{n-1}} \\
	{X^{n-1}} & X & {X/X^{n-1}}
	\arrow[from=1-1, to=1-2]
	\arrow["{f'}"', from=1-1, to=2-1]
	\arrow[from=1-2, to=1-3]
	\arrow["f"', from=1-2, to=2-2]
	\arrow["{\overline{f}}"', from=1-3, to=2-3]
	\arrow[from=2-1, to=2-2]
	\arrow[from=2-2, to=2-3],
\end{tikzcd}\]
and then, due to the Cofibration Axiom, we must have 
\begin{equation*}
\lambda(X,f,\mathcal{F})=\lambda(X/X^{n-1},\overline{f},\pi_\ast\mathcal{F})+\lambda(X^{n-1},f',\mathcal{F})-r.    
\end{equation*}
Now, we also have
\begin{equation*}
    \varLambda_c(X,f,\mathcal{F})=\varLambda_c(X/X^{n-1},\overline{f},\pi_\ast\mathcal{F})+\varLambda_c(X^{n-1},f',\mathcal{F})-r.
\end{equation*}
By induction, we have $\lambda(X^{n-1},f',\mathcal{F})=\varLambda_c(X^{n-1},f',\mathcal{F})$. In addition, as $X/X^{n-1}$ is definable and homeomorphic to a wedge product of spheres by a homeomorphism that sends cells into cells, so, by the Wedge Spheres Axiom, we have
\begin{equation*}
    \lambda(X/X^{n-1},\overline{f},\pi_\ast\mathcal{F})=\varLambda_c(X/X^{n-1},\overline{f},\pi_\ast\mathcal{F}).
\end{equation*}
Then we obtain $\lambda(X,f,\mathcal{F})=\varLambda_c(X;f,\mathcal{F})$.
\end{proof}

 \end{document}